\documentclass[export]{notes}

\declarecommand\Grel[1]{\normalfont(#1$'$)}

\title{Smallest nonabelian quotients of surface braid groups}

\opext{UConf,ord,Int,I}
\ropext{II}

\AtBeginDocument{
  \renewcommand{\ref}{\cref}
}
\begin{document}

\begin{abstract}
   We give a sharp lower bound on the size of nonabelian quotients of the surface braid group $B_n(\Sigma_g)$ and classify all quotients that attain the lower bound: Depending on $n$ and $g$, a quotient of minimum order is either a symmetric group or a 2-step nilpotent $p$-group.
\end{abstract}

\section{Introduction}

The Artin braid group $B_n$ arises as the fundamental group of $\UConf_n(\DD)$, the configuration space of $n$ distinct unordered points on the open disk $\DD$. One can generalize this construction to define for an oriented, closed genus $g$ surface $\Sigma_g$ the \emph{surface braid groups}
$$B_n(\Sigma_g) = \pi_1(\UConf_n(\Sigma_g)).$$

It was shown by Kolay \cite{kolay_smallest_2021} that for $n = 3$ or $n \geq 5$, the smallest noncyclic finite quotient of $B_n$ is the symmetric group $S_n$, in the sense that $S_n$ has minimum order amongst noncyclic quotients of $B_n$ and $S_n$ is the unique noncyclic quotient of $B_n$ of minimum order.

In this paper we consider the analogous question for surface braid groups. With our main result we show that whilst $S_n$ is a quotient of $B_n(\Sigma_g)$, it is not generally the smallest nonabelian quotient, in contrast to the regular braid groups. The new minimal nonabelian quotients that arise are 2-step nilpotent $p$-groups which will be defined in \ref{standard-jn2}.

For $g \geq 1$ there is an embedding $B_n \hookrightarrow B_n(\Sigma_g)$ \cite{birman_braid_1969}. By a \emph{braid-reduced} quotient of $B_n(\Sigma_g)$ we mean a finite quotient with $B_n$ having cyclic image. Our main result is the following theorem.

\begin{theorem}[Smallest nonabelian quotients of $B_n(\Sigma_g)$]\label{big-theorem}
    Let $n \geq 5$ and $g \geq 1$. Suppose that $G$ is a finite nonabelian quotient of $B_n(\Sigma_g)$.
    \begin{parts}
    \item \label{big-theorem-non-braid-reduced}
        If $G$ is not braid-reduced then $|G| \geq n!$ with equality if and only if $G \iso S_n$.
    \item \label{big-theorem-braid-reduced}
        If $G$ is braid-reduced then $G$ is 2-step nilpotent and $|G| \geq p^{2g+j}$, where $p$ is the smallest prime dividing $g+n-1$ and $j = 1$ or $2$ according to whether $p$ is odd or 2 respectively. Equality occurs if and only if either $G \iso \I(p^j,g)$ or $G \iso \II(p^j,g)$ (these two groups are nonisomorphic 2-step nilpotent $p$-groups defined in \ref{standard-jn2}).
    \end{parts}
    In particular the smallest non-nilpotent quotient of $B_n(\Sigma_g)$ is $S_n$.
\end{theorem}


Note that \ref{big-theorem} implies the following qualitative result.

\begin{corollary}\ 
    \begin{parts}
    \item
        Fix $g \geq 1$. For all sufficiently large $n$, the smallest nonabelian quotients of $B_n(\Sigma_g)$ are 2-step nilpotent $p$-groups (in particular, the smallest nonabelian quotient is not $S_n$). 
    \item
        Fix $n \geq 5$. For all sufficiently large $g$, the smallest nonabelian quotient of $B_n(\Sigma_g)$ is $S_n$. Also, there exists a (small) $g$ for which this is not true.
    \end{parts}
\end{corollary}

\newpage
\begin{remarks}[Smaller cases]
\item
    If $n = 1,2,3,4$ and $g \geq 1$ (with the exception of $(n,g) = (1,1)$ where $B_n(\Sigma_g) = \pi_1(T^2) = \ZZ^2$ is abelian) then the symmetric group $S_3$ is the smallest nonabelian quotient of $B_n(\Sigma_g)$.
\item
    If $g = 0$ then $B_n(\Sigma_g)$ is the spherical braid group $B_n(S^2)$ which is an intermediate quotient of the map $B_n \to S_n$ \cite{fadell_braid_1961}. It follows from the result of Kolay \cite{kolay_smallest_2021} that the smallest quotient of $B_n(S^2)$ is $S_n$ for $n \geq 5$ and $S_3$ for $n = 3,4$. For $n = 1$ and $2$ we note that $B_n(S^2)$ is abelian.
\end{remarks}

From \ref{big-theorem} we obtain partial confirmation of a conjecture of Chen \cite[Conjecture 1.3]{chen_surjective_2019}:

\begin{corollary}\label{no-surjections}
    Let $n \geq 5$ and $m \geq 3$, and let $g,h \geq 0$. If $n > m$ then there are no surjective homomorphisms $$B_n(\Sigma_g) \to B_m(\Sigma_h).$$
\end{corollary}

\paragraph{Proof method.} 
\ref{big-theorem-non-braid-reduced} follows from Kolay: By mapping a braid to its permutation on points, $S_n$ is a finite quotient of $B_n(\Sigma_g)$. If $B_n \to B_n(\Sigma_g) \to G$ has noncyclic image then $|G| \geq n!$ with the bound attained only by $G \iso S_n$.

The primary contribution of this paper is \ref{big-theorem-braid-reduced}, which considers the braid-reduced quotients. We utilize a presentation of $B_n(\Sigma_g)$ (\ref{presentation}) due to Bellingeri \cite{bellingeri_presentations_2004} and assume that $B_n$ has cyclic image to reduce the relations and conclude that a braid-reduced quotient $G$ must be nilpotent. If we further assume that $G$ is a nonabelian braid-reduced quotient of minimum order then $G$ belongs to a class of nilpotent groups called JN2 groups (\ref{def:jn2}) which were classified by Newman in 1960 \cite{newman_class_1960}. It then suffices to find the smallest JN2 groups which can be realized as a quotient of $B_n(\Sigma_g)$, a straightforward task given the concrete nature of Newman's classification.

\ref{sec:jn2} provides a self-contained exposition of the classification of JN2 groups. In \ref{sec:proof} we prove \ref{big-theorem-braid-reduced}, as well as \ref{no-surjections}.

\paragraph{Acknowledgements.} I am grateful to my advisor Benson Farb for continued support throughout this project and for detailed comments on many revisions of this paper, as well as for suggesting this problem in the first place. I thank Peter Huxford for useful discussions about braid groups and small $p$-groups, and for many helpful suggestions during the editing process. I also thank Dan Margalit for taking the time to read and comment on an earlier draft.

\section{Just 2-step nilpotent groups}\label{sec:jn2}

In this section we introduce and classify JN2 groups, a class of nilpotent groups which includes all minimal nonabelian braid-reduced quotients of $B_n(\Sigma_g)$.

\begin{definition}\label{def:jn2}
    A group $G$ is \i{just 2-step nilpotent} (JN2) if $G$ is 2-step nilpotent (in particular, nonabelian) and every proper quotient of $G$ is abelian.
\end{definition}

Finite JN2 groups admit a complete and explicit classification due to Newman \cite{newman_class_1960}: Any finite JN2 group can be assigned a unique class $(p^j,m)$ where $p$ is a prime and $j$ and $m$ are positive integers; up to isomorphism, there are precisely two JN2 groups of a given class $(p^j,m)$. We will state and prove this classification theorem in \ref{jn2-classification}, following the general ideas of \cite{newman_class_1960}.

All JN2 groups will hereafter be assumed to be finite. The following proposition will allow us to define the class $(p^j,m)$ of a JN2 group.

\begin{proposition}[Characterization of JN2 groups {\cite[Theorem 1]{newman_class_1960}}] \label{jn2-characterization}
    A finite group $G$ is JN2 if and only if there exists a prime $p$ such that
    \begin{parts}
    \item $G' \defeq [G,G]$ is cyclic of order $p$,
    \item the center $ZG$ is cyclic of order a power of $p$, and
    \item $G/ZG$ is elementary abelian of exponent $p$. \label{jn2-characterization-gzg}
    \end{parts}
    In particular, a JN2 group is a $p$-group.
\end{proposition}

\begin{proof}\
\begin{proof-iff}
    \pforwards

    Let $G$ be a finite JN2 group. For every nontrivial normal subgroup $N \trianglelefteq G$, we have that $G' \leq N$ since any proper quotient of $G$ is abelian. Since $G$ is 2-step nilpotent, $G' \leq ZG$. Consequently,
    \begin{parts}
    \item
        $G'$ is abelian and admits no proper nontrivial subgroups so $G' \iso \ZZ/p\ZZ$ for some prime $p$.
    \item
        $ZG$ cannot be properly decomposed as a direct sum: Any nontrivial subgroup of $ZG$ contains $G'$ so no two nontrivial subgroups intersect trivially. Since $ZG$ is finite abelian, it must be cyclic of prime power order. The prime must be $p$ because $G' \leq ZG$.
    \item
        $G/ZG$ is abelian because $G' \leq ZG$. For $x,y \in G$, we have that $[x^p,y] = [x,y]^p$ by using the identity $$[xz,y] = z[x,y]z^{-1}[z,y]$$ and noting that $[x,y]$ is central because $G' \leq ZG$. But $G'$ has order $p$, so in fact $[x^p,y] = 1$. Thus $x^p \in ZG$ for all $x \in G$, which is to say that $G / ZG$ has exponent $p$.
    \end{parts}

    \pbackwards
    Suppose $G$ is a finite group satisfying (a), (b), and (c). Then $G' \neq \{1\}$ by (a) and $G' \leq ZG$ by (c) so $G$ is 2-step nilpotent.

    If $N \trianglelefteq G$ is a normal subgroup with $G' \not \leq N$ then $N \cap G' = \{1\}$ by (a). Since $N$ is normal, $[N,G] \leq N \cap G' = \{1\}$ so $N \leq ZG$. But $G' \leq ZG$, and (a) and (b) imply that any nontrivial subgroup of $ZG$ intersects $G'$ nontrivially. Thus $N = \{1\}$. We conclude that every proper quotient of $G$ is abelian. \qedhere
\end{proof-iff}
\end{proof}

An immediate corollary of \ref{jn2-characterization-gzg} is that $V \defeq G/ZG$ has the structure of an $\FF_p$-vector space. Note that vector addition in $V$ is written multiplicatively and scalar multiplication of an element $x \mod ZG \in V$  by a scalar $r \in \FF_p$ is written as
\begin{gather*}
    r \cdot (x \mod ZG) = x^r \mod ZG.
\end{gather*}
Fix a generator $z$ of $ZG$. This fixes a generator $z^{p^{j-1}}$ of $G'$ and hence an identification of $G'$ with $\FF_p$. Define a pairing
\begin{align*}
   V \x V &\to G' = \FF_p \\
  (x \mod ZG, y \mod ZG) &\mapsto [x,y]
\end{align*}
This pairing is a well-defined, bilinear, nondegenerate, alternating form which makes $V$ into a symplectic vector space. In particular, $\dim V$ is even.

Thus associated to each JN2 group $G$ is a \emph{class} $(p^j, m)$ where $|ZG| = p^j$ and $\dim V = 2m$, so $G$ fits into the short exact sequence
\[
    \begin{tikzcd}
        1 \rar & \ZZ/p^j\ZZ \rar & G \rar & (\ZZ/p\ZZ)^{2m} \rar & 0.
    \end{tikzcd}
\]

The symplectic structure on central factor groups $V = G/ZG$ is key to the classification theorem because symplectic automorphisms on central factor groups can  be used to construct isomorphisms between certain JN2 groups of the same class. The following lemma extracts from a JN2 group a normalized symplectic basis on its associated vector space $V$.
\begin{lemma} \label{symplectic-basis-normalized}
    Let $G$ be JN2 of class $(p^j,m)$ where $p^j \neq 2$, with a fixed generator $z$ of $ZG$. Then there exists a symplectic basis $\sB = \{a_i \mod ZG, b_i \mod ZG\}_{i=1}^m$ of $V = G/ZG$ such that the representatives $a_i,b_i \in G$ satisfy either
    \begin{parts}[label=(\Roman*), font=\normalfont]
    \item
        $a_i^p = b_i^p = 1$ for all $i$, or
    \item
        $a_1^p = b_1^p = z$ and $a_i^p = b_i^p = 1$ for $2 \leq i \leq m$.
    \end{parts}
    We will say that $\sB$ is type I or II accordingly.
\end{lemma}

\begin{remark}[Nomenclature]
    For the reader familiar with existing terminology from \cite{newman_class_1960}, a ``type I (II) basis'' as named in our \ref{symplectic-basis-normalized} corresponds to a ``canonic normal basis with zero (one) pairs of type II'' in the vocabulary of Newman.
\end{remark}

\begin{proof}
    Note that $x^p \in ZG$ for all $x \in G$ because $G/ZG$ has exponent $p$. Let $(ZG)^p = \{ u^p : u \in ZG \}$ and identify $ZG / (ZG)^p$ with $\FF_p$ by mapping $z \mod (ZG)^p \mapsto 1$. Define a map
    \begin{align*}
        \nu: V &\to ZG / (ZG)^p = \FF_p \\
        x \mod ZG &\mapsto x^p \mod (ZG)^p
    \end{align*}
    Viewing $V = G/ZG$ as a vector space written multiplicatively, $\nu$ commutes with scalar multiplication and
    \begin{gather*}
        \nu((x \mod ZG) (y \mod ZG)) = (xy)^p \mod (ZG)^p = [y,x]^{\frac{p(p-1)}{2}} x^p y^p \mod (ZG)^p
    \end{gather*}
    for $x,y \in G$ so $\nu$ is a linear functional as long as $[y,x]^{\frac{p(p-1)}{2}} = 1 \mod (ZG)^p$. This holds if $p^j \neq 2$: If $p$ is odd then $p \mid \frac{p(p-1)}{2}$ so $[y,x]^\frac{p(p-1)}{2} = 1$ because $G'$ has order $p$. If $j \geq 2$ then $G' \nleq ZG$ so $G' \leq (ZG)^p$.

    If $\nu$ is the trivial linear functional on $V$, take $\sB$ to be any symplectic basis of $V$. Otherwise, there exists a symplectic basis $\sB$ of $V$ such that $\nu$ written with respect to $\sB$ is the row vector
    \begin{gather*}
        \nu = \m{1 & 1 & 0 & \cdots & 0}
    \end{gather*}
    because symplectic automorphisms act transitively on nontrivial vectors.
    
    In other words, for each basis vector $x_j \mod ZG \in \sB$,
    $$x_j^p \mod (ZG)^p = \nu(x_j \mod ZG) = z^{\nu_j} \mod (ZG)^p$$
    so there exists $u_j \in ZG$ such that $x_j^p = z^{\nu_j} u_j^p$. Then $x_ju_j^{-1} \equiv x_j \mod ZG$ and $(x_ju_j^{-1})^p = z^{\nu_j}$. Thus $x_ju_j^{-1} \in G$ are representatives of the basis $\sB$ satisfying (I) if $\nu$ is trivial and (II) otherwise.
\end{proof}

We will now construct two standard non-isomorphic JN2 groups for each given class $(p^j,m)$. The proof of the classification theorem will exhibit an isomorphism from any arbitrary JN2 group to a standard one. The primary method of constructing larger JN2 groups from smaller ones is taking a central product.

\begin{definition}[Central product]\label{central-product}
    Let $G$ and $H$ be groups for which there exists an isomorphism $\varphi: ZG \to ZH$. Define the \i{central product} of $G$ and $H$ (with respect to $\varphi$) to be
    \begin{gather*}
        G \odot H = (G \x H) / N
    \end{gather*}
    where $N = \ang{(g,\varphi(g)^{-1}): g \in ZG}$, namely identifying $ZG \x 1$ with $1 \x ZH$ by the isomorphism $\varphi$. By $G^{\odot n}$ we mean the central product of $n$ copies of $G$ with the identity isomorphism on $ZG$.

    Note that if $G,H$ are JN2 of class $(p^j,m_1)$ and $(p^j,m_2)$ then $G \odot H$ is JN2 of class $(p^j, m_1 + m_2)$ by \ref{jn2-characterization} since
    \begin{enumerate}
    \item
        $(G \odot H)' = G' \x H' / N \iso G' \iso H' \iso \ZZ/p\ZZ$,
    \item
        $Z(G \odot H) \iso ZG \iso ZH$, and
    \item
        $(G\odot H)/Z(G \odot H) \iso (G/ZG) \x (H / ZG)$.
    \end{enumerate}
\end{definition}

\begin{construction}[Standard JN2 groups]\label{standard-jn2}
    Define the groups
    \begin{align*}
      M(p^j) &= \ang{z, a, b: [z,a] = [z,b] = 1; [a,b] = z^{p^{j-1}}; z^{p^j} = a^p = b^p = 1} \\
      N(p^j) &= \ang{z, a, b: [z,a] = [z,b] = 1; [a,b] = z^{p^{j-1}}; z^{p^j} = 1; a^p = b^p = z} \\
      \I(p^j,m) &= M(p^j)^{\odot m} \\
      \II(p^j,m) &= N(p^j) \odot M(p^j)^{\odot (m-1)}
    \end{align*}
    Observe the following:
    \begin{enumerate}
    \item
        $M(p^j)$ and $N(p^j)$ are JN2 (by \ref{jn2-characterization}) of class $(p^j,1)$ with each center generated by $z$ and $\{a,b\}$ as a symplectic basis of $V$.
    \item
        $\I(p^j,m)$ and $\II(p^j,m)$ are JN2 of class $(p^j,m)$ by the remarks following \ref{central-product}.
    \item
        $\I(p^j,m)$ and $\II(p^j,m)$ are not isomorphic when $p^j \neq 2$: The group $N(p^j)$ has an element of order $p^{j+1}$ (for example, $a$ or $b$) and therefore so does $\II(p^j,m)$. On the contrary, the group $M(p^j)$, and consequently also $\I(p^j,m) = M(p^j)^{\odot m}$, has exponent at most $p^j$: The linear functional $\nu$ (as in the proof of \ref{symplectic-basis-normalized}) is trivial on the symplectic basis $\{a,b\}$ so $M(p^j)^p \leq (ZG)^p$, hence $M(p^j)^{p^j} \leq (ZG)^{p^j} = 1$.

        \emph{Note:} If $p^j = 2$, then $\I(p^j,m)$ and $\II(p^j,m)$ are still non-isomorphic: $M(2)$ is the dihedral group $D_8$ and $N(2)$ is the quaternion group $Q_8$, which contain two and six elements of order 4 respectively and both have centers of order 2. In particular no elements of order 4 are central. The larger groups $\I(p^j,m)$ and $\II(p^j,m)$ can then be distinguished by counting the number of elements of order 4 because only central elements are identified in the central product. We will not require this case.
    \end{enumerate}
\end{construction}

We are now ready to state and prove the classification theorem of JN2 groups.

\begin{theorem}[Classification of finite JN2 groups {\cite[Theorem 5, Theorem 7(c), Lemma 8(i)]{newman_class_1960}}] \label{jn2-classification}
    Let $G$ be JN2 of class $(p^j,m)$. Suppose that $p^j \neq 2$. Then $G$ is isomorphic to either $\I(p^j,m)$ or $\II(p^j,m)$.
\end{theorem}

\begin{proof}
    Let $z$ be a generator of $ZG$ and let $\sB$ be the symplectic basis given by \ref{symplectic-basis-normalized}. In the notation of \ref{symplectic-basis-normalized}, let $H_i = \ang{z,a_i,b_i}$. If $\sB$ is type I then $H_i = M(p^j)$ for all $i$. If $\sB$ is type II then $H_1 = N(p^j)$ and $H_i = M(p^j)$ for $i \geq 2$.

    The subgroups $H_i$ commute pairwise, together generate $G$, and intersect precisely in their centres $\ang{z}$, so $G \iso \bigodot_{i=1}^m H_i$. Hence $G$ is isomorphic to $\I(p^j,m)$ or $\II(p^j,m)$ according to the type of the basis $\sB$.
\end{proof}

\begin{remarks}
\item
    (Generalizations)
    For brevity, we have excluded the case of $p^j = 2$ and specialized to finite groups. With additional work, the $p^j = 2$ case and some infinite JN2 groups (those with a countable symplectic basis) also admit a classification as central products of elementary JN2 groups, see \cite{newman_class_1960}.
\item
    (Special cases)
    Note that $M(p)$ and $N(p)$ are the only two groups of order $p^3$. The group $M(p) = \I(p,1)$ is isomorphic to the Heisenberg group over $\FF_p$. A generalization of the finite Heisenberg groups are the \emph{extraspecial groups}, which are defined to be $p$-groups $G$ with $ZG$ order $p$ and $G/ZG$ nontrivial elementary abelian. In particular, extraspecial groups are JN2 and it follows from \ref{jn2-classification} that there are precisely two distinct extraspecial groups of order $p^{1+2m}$ for each choice of a prime $p$ and positive integer $m$, and that this exhausts all extraspecial groups.
\end{remarks}

\section{Minimal nonabelian quotients of $B_n(\Sigma_g)$}\label{sec:proof}

In this section we provide the proof of \ref{big-theorem-braid-reduced}.  
%
%
The strategy of the proof will be to utilize an explicit presentation of the surface braid groups (\ref{presentation}) to characterize braid-reduced quotients by the relations that they must satisfy (\ref{br-gen-set}). We will then show that many JN2 groups are realized as nonabelian braid-reduced quotients of $B_n(\Sigma_g)$ (\ref{small-br-quotients}) and finally prove that all nonabelian braid-reduced quotients of minimum order belong to the list of JN2 groups in \ref{small-br-quotients}.

The following presentation of $B_n(\Sigma_g)$ is due to Bellingeri \cite{bellingeri_presentations_2004}.

\begin{theorem}[Presentation of $B_n(\Sigma_g)$ {\cite[Theorem 1.2]{bellingeri_presentations_2004}}] \label{presentation}
    For $g \geq 1$ and $n \geq 2$, the surface braid group $B_n(\Sigma_g)$ admits the presentation:
    \begin{itemize}
    \item
        generators: $\sigma_1,\dots,\sigma_{n-1},a_1,\dots,a_g,b_1,\dots,b_g$.
    \item
        relations:

        \hspace{1.5em} braid relations:
        \begin{alignat*}{3}
            &&&[\sigma_i, \sigma_j] = 1 &&(1 \leq i,j \leq n-1 \text{ and } |i-j| \geq 2) \\
            &&&\sigma_i\sigma_{i+1}\sigma_i = \sigma_{i+1}\sigma_i\sigma_{i+1} &&(1 \leq i \leq n-2) \\
            \intertext{\hspace{1.5em} mixed relations:}
            &\text{\normalfont(R1)}\quad
            && [a_r, \sigma_i] = [b_r, \sigma_i] = 1
            && (1 \leq r \leq g \text{ and } i \neq 1) \\
            &\text{\normalfont(R2)}
            && [a_r, \sigma_1^{-1}a_r\sigma_1^{-1}] = [b_r, \sigma_1^{-1}b_r\sigma_1^{-1}] = 1 
            && (1 \leq r \leq g) \\
            &\text{\normalfont(R3)}
            && [a_s,\sigma_1a_r\sigma_1^{-1}] = [b_s,\sigma_1b_r\sigma_1^{-1}] = 1   
            && (1 \leq s < r \leq g) \\
            &&&[b_s,\sigma_1a_r\sigma_1^{-1}] = [a_s, \sigma_1b_r\sigma_1^{-1}] = 1 
            && (1 \leq s < r \leq g) \\
            &\text{\normalfont(R4)}
            && [a_r, \sigma_1^{-1}b_r\sigma_1^{-1}] = \sigma_1^2
            && (1 \leq r \leq g) \\
            &\text{\normalfont(TR)}
            && [a_1,b_1^{-1}] \cdots [a_g,b_g^{-1}] = \sigma_1\sigma_2\cdots \sigma_{n-1}^2 \cdots \sigma_2\sigma_1
        \end{alignat*}
    \end{itemize}
\end{theorem}

\begin{remark}[Geometric interpretation of the presentation]\label{presentation-geometric-intuition}
    The embeddings $B_n \hookrightarrow B_n(\Sigma_g)$ identify the Artin braid generators with the Bellingeri generators $\sigma_i$. The remaining generators $a_r,b_r$ can be understood loosely to be the standard generators of $\pi_1(\Sigma_g)$.

    More precisely, let $\{p_1,\dots,p_n\} \in \UConf_n(\Sigma_g)$ denote the basepoint of $B_n(\Sigma_g)$ and let $D \subset \Sigma_g$ be an open disk with $p_1 \in \del D$, with $p_2,\dots,p_n$ in the interior of $D$. There is an inclusion
    $$ \pi_1(\Sigma_g - D, p_1) \hookrightarrow B_n(\Sigma_g) $$
    which takes a loop $\gamma$ in $\Sigma_g - D$ to the braid on $\Sigma_g$ with first strand $\gamma$ and all other strands trivial. The group $\pi_1(\Sigma_g - D, p_1)$ is free on $2g$ generators and surjects onto $\pi_1(\Sigma_g, p_1)$ which has a standard presentation. The surface braid group generators $a_r,b_r \in B_n(\Sigma_g)$ can then be understood as a choice of a free generating set of $\pi_1(\Sigma_g - D, p_1)$ which lifts the standard generating set of $\pi_1(\Sigma_g, p_1)$. It should be emphasized that the lifts are not canonical and that the presentation depends on the choices; the curious reader may refer to \cite{bellingeri_presentations_2004} for illustrations of the loops which produce this particular presentation.
\end{remark}

\newpage
\begin{lemma}[Characterization of braid-reduced quotients] \label{br-gen-set}
    Let $n \geq 3$ and $g \geq 1$. A finite group $G$ is a braid-reduced quotient of $B_n(\Sigma_g)$ if and only if $G$ admits a generating set $\{\sigma,a_1,b_1,\dots,a_g,b_g\}$ satisfying the relations
    \begin{alignat*}{3}
      &\text{\Grel{R1}}\quad
      && [a_r, \sigma] = [b_r, \sigma] = 1
      && (1 \leq r \leq g) \\
      &\text{\Grel{R3}}
      && [a_s,a_r] = [b_s,b_r] = [b_s,a_r] = [a_s, b_r] = 1\quad  
      && (1 \leq s < r \leq g) \\
      &\text{\Grel{R4}}
      && [a_r, b_r] = \sigma^2
      && (1 \leq r \leq g) \\
      &\text{\Grel{TR}}
      && \sigma^{2(g+n-1)} = 1 
   \end{alignat*}
\end{lemma}

\begin{proof}
    A finite quotient of $B_n(\Sigma_g)$ is presented by \ref{presentation} with additional relations. The condition that $B_n$ has cyclic image in a quotient is equivalent to adding the relations $$\sigma_i = \sigma_1,\qquad 1\leq i \leq n.$$ If we add these relations and write $\sigma = \sigma_1$, the relation (R2) is made redundant and (R1), (R3), and (R4) respectively reduce to the relations \Grel{R1}, \Grel{R3}, and \Grel{R4} as in the statement of the lemma. The final relation (TR) reduces to
    \begin{gather*}
        [a_1,b_1^{-1}] \cdots [a_g,b_g^{-1}] = \sigma^{2(n-1)}
    \end{gather*}
    which is equivalent to \Grel{TR} because from \Grel{R4} we can write $a_r = b_r^{-1} \sigma^{-2}a_r b_r$ so that
    \begin{gather*}
        [a_r, b_r^{-1}]
        = a_r b_r^{-1} a_r^{-1} b_r
        \overset{\text{\Grel{R4}}}{=}
        (b_r^{-1} \sigma^{-2} a_r b_r) b_r^{-1} a_r^{-1} b_r
        = b_r^{-1} \sigma^{-2} b_r
        \overset{\text{\Grel{R1}}}{=}
        \sigma^{-2}. \qedhere
    \end{gather*}
\end{proof}

The following lemma proves that many JN2 groups are braid-reduced quotients.

\begin{lemma}\label{small-br-quotients}
    Let $n \geq 3$ and $g \geq 1$. Let $p$ be a prime dividing $g+n-1$.
    \begin{parts}
    \item
        If $p = 2$ then $\I(2^2,g)$ and $\II(2^j,g)$ for all $j \geq 2$ are nonabelian braid-reduced quotients of $B_n(\Sigma_g)$.
    \item
        If $p$ is odd then $\I(p,g)$ and $\II(p^j,g)$ for all $j \geq 1$ are nonabelian braid-reduced quotients of $B_n(\Sigma_g)$.
    \end{parts}
\end{lemma}

\begin{proof}
    Let $p$ be a prime dividing $g+n-1$. By \ref{br-gen-set} we need to exhibit a generating set $\{\sigma,a_r,b_r\}$ of each group satisfying relations \Grel{R1}, \Grel{R3}, \Grel{R4}, and \Grel{TR}.

    In any of the JN2 groups in the statement of the theorem, fix a generator $z$ of the center and choose $a_1,b_1,\dots,a_g,b_g$ to be the representatives of a symplectic basis of $V$ given by \ref{symplectic-basis-normalized}. By \ref{jn2-classification} this basis will be type I for $\I(2^2,g)$ and $\I(p,g)$, and type II for $\II(2^j,g)$ and $\II(p^j,g)$. Note that with the given symplectic form, the condition that a basis is symplectic is simply that all basis elements commute except symplectic pairs $[a_r,b_r] = z^{p^{j-1}}$. In particular, \Grel{R3} is satisfied.

    We will now choose $\sigma$ for each group and verify that $\{\sigma,a_r,b_r\}$ generate the group and satisfy \Grel{R4}. 
    \begin{enumerate}
    \item
        $\I(2^2,g)$ is generated by $\sigma = z$ and the $a_r,b_r$. These generators satisfy \Grel{R4} because $[a_r,b_r] = z^2 = \sigma^2$.
    \item
        $\II(2^j,g)$, for a given $j \geq 2$, is generated by the $a_r,b_r$ alone because $a_1^p = z$. If we choose $\sigma = z^{2^{j-2}}$ then \Grel{R4} is satisfied because $[a_r, b_r] = z^{2^{j-1}} = \sigma^2$.
    \item
        $\I(p,g)$ for odd prime $p$ is generated by $\sigma = z^{(p^j+1)/2}$ and the $a_r,b_r$. Then \Grel{R4} is satisfied because $[a_r,b_r] = z = \sigma^2$.
    \item
        $\II(p^j,g)$, for given odd prime $p$ and $j \geq 1$, is generated by $a_r,b_r$ alone because $a_1^p = z$. Then set $\sigma = z^{(p^j+p^{j-1})/2}$ so that \Grel{R4} is satisfied because $[a_r,b_r] = z^{p^{j-1}} = \sigma^2$.
    \end{enumerate}
    In all cases $\sigma$ was chosen to be central, hence \Grel{R1} is satisfied.
    
    It remains to check that \Grel{TR} holds, namely that $|\sigma|$ divides $2 (g+n-1)$. Recall that we are assuming that $p \mid (g+n-1)$. In cases 1 and 2, we have $p = 2$ and $|\sigma| = 4 = 2p \mid 2(g+n-1)$. In cases 3 and 4, we have $|\sigma| = p \mid (g+n-1)$. 
\end{proof}

We are now prepared to prove \ref{big-theorem-braid-reduced}.

\begin{proof}[Proof of \ref{big-theorem-braid-reduced}]
    Let $G$ be a nonabelian braid-reduced quotient and let $\{\sigma,a_1,b_1,\dots,a_g,b_g\}$ denote the generating set of $G$ as given by \ref{br-gen-set}. By \Grel{R1} and \Grel{R3}, all pairs of these generators commute except for pairs $a_r,b_r$ so $G' = \ang{\sigma^2}$ by \Grel{R4}. Then $G'$ is central and nontrivial, which is to say that $G$ is 2-step nilpotent.

    Assume now that $G$ is of minimum order. Then $G$ has no proper nonabelian quotients and thus is JN2 of some class $(p^j,m)$.

    We make three claims:
    \begin{enumerate}
    \item $p^j \neq 2$,
    \item $m = g$, and
    \item $p \mid (g+n-1)$.
    \end{enumerate}

    These claims will complete the proof: Since $|G| = p^{2m+j}$, it follows from claims 1 and 2 that $|G| \geq p^{2g+1}$ if $p$ is odd and $|G| \geq 2^{2g+2}$ if $p = 2$. Claims 1 and 3 along with the minimality of $G$ together imply that $G$ is one of (in particular, the smallest of) the quotients constructed in \ref{small-br-quotients}. Explicitly: If $g+n-1$ is even then $p^j = 2^2$. Otherwise $p^j = p$ where $p$ is the smallest prime dividing $g+n-1$. Finally, $G$ must be isomorphic to either $\I(p^j,g)$ or $\II(p^j,g)$ by \ref{jn2-classification}.

    \underline{Proof of claims}:
    Let $d = |\sigma|$. By \Grel{R1}, $\sigma$ is central so $d \mid p^j$. But $p = |G'| = |\sigma^2|$ so $d \mid 2p$. Thus either $p$ is odd and $p = d$, or $p = 2$ and $d = 4$.

    \begin{enumerate}
    \item
        If $p$ is odd then $p^j \neq 2$. If $p = 2$ then $p^j \geq d = 4$ so $p^j \neq 2$.
    \item
        We will show that $\dim V = 2g$ by proving that 
        $$\sB = \{a_r \mod ZG, b_r \mod ZG\}_{r=1}^g$$
        is a basis of $V$. Every element $x \in G$ can be written uniquely in the form
        $$x = \sigma^k a_1^{i_1} \cdots a_g^{i_g} b_1^{j_1} \cdots b_g^{j_g}$$
        using commuting relations \Grel{R1}, \Grel{R3}, \Grel{R4} so $\sB$ is a generating set. To prove that $\sB$ is linearly independent, let
        $$y = a_1^{i_1} \cdots a_g^{i_g} b_1^{j_1} \cdots b_g^{j_g} \in G$$
        and suppose that $y = 0 \mod ZG$, which is to suppose that an arbitrary linear combination of elements of $\sB$ is trivial in $V$. Then $y$ is central so
        $$[y,b_1] = [a_1^{i_1},b_1] = \sigma^{-2i_1} = 1$$
        which implies that $d \mid 2 i_1$ and thus $i_1 = 0 \mod p$: If $p$ is odd then $d = p$ so $p \mid i_1$. If $p = 2$ then $d = 4 \mid 2 i_1$ so $p = 2 \mid i_1$.

        Similarly $i_r = j_r = 0 \mod p$ for all $r$, which is to say that all coefficients of the linear combination are trivial over the base field $\FF_p$. This proves the linear independence of $\sB$.
    \item
        The relation \Grel{TR} imposes the relation $d \mid 2(g+n-1)$. Either $d = p$ is odd or $d = 4$ and $p = 2$; in both cases \Grel{TR} implies that $p \mid (g+n-1)$. \qedhere
    \end{enumerate}
\end{proof}

\begin{proof}[Proof of \ref{no-surjections}]
    Let $n \geq 5$ and $m \geq 3$, and let $g,h \geq 0$. If there is a surjection $B_n(\Sigma_g) \to B_m(\Sigma_h)$ then the composition $B_n(\Sigma_g) \to B_m(\Sigma_h) \to S_m$ is also surjective. Since $S_m$ is not nilpotent when $m \geq 3$, we must have $m \geq n$.
\end{proof}

\begin{remark}[Punctured surfaces, surfaces with boundary]
    In his paper \cite{bellingeri_presentations_2004}, Bellingeri also gives a presentation of the braid group of a genus $g$ surface with $m$ punctures (equivalently for the purposes of braid groups, $m$ boundary components). The above methods can be used nearly verbatim to prove that the smallest nonabelian quotient of $B_n(\Sigma_{g,m})$ is the smaller of $S_n$ or $\I(2^2,g)$ and $\II(2^2,g)$.

\end{remark}

{\small\bibliography{025/025}}

\begin{thebibliography}{1}

\bibitem{bellingeri_presentations_2004}
P.~Bellingeri.
\newblock On presentations of surface braid groups.
\newblock {\em Journal of Algebra}, 274(2):543--563, Apr. 2004.

\bibitem{birman_braid_1969}
J.~S. Birman.
\newblock On braid groups.
\newblock {\em Communications on Pure and Applied Mathematics}, 22(1):41--72,
  Jan. 1969.

\bibitem{chen_surjective_2019}
L.~Chen.
\newblock Surjective homomorphisms between surface braid groups.
\newblock {\em Israel Journal of Mathematics}, 232(1):483--500, Aug. 2019.

\bibitem{fadell_braid_1961}
E.~Fadell and J.~Van~Buskirk.
\newblock On the braid groups of {$E^2$} and {$S^2$}.
\newblock {\em Bulletin of the American Mathematical Society}, 67(2):211--213,
  1961.

\bibitem{kolay_smallest_2021}
S.~Kolay.
\newblock Smallest non-cyclic quotients of braid and mapping class groups, Oct.
  2021.
\newblock arXiv: 2110.02162.

\bibitem{newman_class_1960}
M.~F. Newman.
\newblock On a {Class} of {Nilpotent} {Groups}.
\newblock {\em Proceedings of the London Mathematical Society},
  s3-10(1):365--375, 1960.

\end{thebibliography}

\end{document}